\documentclass[preprint,12pt,sort&compress,numbers,draft]{elsarticle}

\usepackage{amsmath,amsthm,amsfonts,amssymb,latexsym,mathrsfs,url,enumerate}

\newtheorem{thm}{Theorem}[section]
\newtheorem{lem}[thm]{Lemma}
\newtheorem{coro}[thm]{Corollary}
\newtheorem{prop}[thm]{Proposition}

\newtheorem*{WI}{Weyl's Inequality}

\theoremstyle{definition}
\newtheorem{definition}[thm]{Definition}

\newtheorem{rem}[thm]{Remark}

\numberwithin{equation}{section}

\def\R{\mathfrak{R}}

\def\diag{{\rm diag}}
\def\sgn{{\rm sign}}

\newcommand{\intrz}{\preccurlyeq_{\rm int}}
\newcommand{\bint}[2]{\sideset{_{#1}}{_{#2}}{\mathop{\sim}}}
\newcommand{\bb}[2]{\sideset{_{#1}}{_{#2}}{\mathop{\approx}}}
\newcommand{\s}[1]{\sum_{#1}}

\def\la{\lambda}

\def\n{\mathfrak{n}}
\def\H{\mathcal{H}}
\def\RM{\R^{(1)}}

\journal{AAM}

\begin{document}

\begin{frontmatter}

\title{The converse of Weyl's eigenvalue inequality}
\author{Yi Wang\corref{cor1}}
\ead{wangyi@dlut.edu.cn}
\cortext[cor1]{Corresponding author.}
\author{Sainan Zheng\corref{cor2}}
\ead{zhengsainandlut@hotmail.com}

\address{School of Mathematical Sciences, Dalian University of Technology, Dalian 116024, China}

\begin{abstract}
We establish the converse of Weyl's eigenvalue inequality
for additive Hermitian perturbations of a Hermitian matrix.
\end{abstract}

\begin{keyword}
Hermitian matrix\sep eigenvalue\sep Weyl's inequality\sep inertia index
\MSC[2010]  15A42\sep 15B57\sep 26C10
\end{keyword}

\end{frontmatter}

\section{Introduction}

Let $A=(a_{ij})_{m\times n}$ be a complex matrix.
The {\it conjugate transpose} $A^*=(b_{ij})_{n\times m}$ of $A$ is defined by $b_{ij}=\overline{a_{ji}}$.
An $n\times n$ complex square matrix $A$ is called {\it Hermitian} if $A^*=A$.
It is well known that the eigenvalues of a Hermitian matrix are all real.
Throughout this paper we adopt the convention that they always arranged in non-increasing order:
$$\la_1(A)\ge\la_2(A)\ge\cdots\ge\la_n(A).$$
We also set $\la_i(A)=+\infty$ for $i<1$ and $\la_i(A)=-\infty$ for $i>n$ for convenience.

In 1912 Hermann Weyl \cite{Wey12} posed the following problem:
given the eigenvalues of two $n\times n$ Hermitian matrices $A$ and $B$,
how does one determine all possible sets of eigenvalues of the sum $A+B$?
He gave partial answers:
\begin{equation}\label{wi<}
\la_{i+j-1}(A+B)\le\la_{i}(A)+\la_{j}(B).
\end{equation}
Note that $\la_i(-C)=-\la_{n+1-i}(C)$ for any $n\times n$ Hermitian matrix $C$.
Hence \eqref{wi<} is equivalent to
\begin{equation}\label{wi>}
\la_{i+j-n}(A+B)\ge\la_{i}(A)+\la_j(B).
\end{equation}

Although Weyl's problem has been completely solved
(we refer the reader to Allen Knutson and Terence Tao's survey article \cite{KT01} on this topic),
Weyl's inequality \eqref{wi<} is still the source of a great many eigenvalue inequalities
(see \cite[\S 4.3]{HJ13} for instance).
Let $n_+(B)$ ($n_-(B)$, resp.) denote the positive (negative, resp.) index of inertia of $B$.
Then $n_+(B)$ ($n_-(B)$, resp.) is the number of positive (negative, resp.) eigenvalues of $B$
by Sylvester's law of inertia.
Denote $n_+(B)=p$ and $n_-(B)=q$.
Then by \eqref{wi<} and \eqref{wi>},
\begin{equation}\label{weyl-1}
\la_{i+q}(A)\le\la_i(A+B)\le\la_{i-p}(A).
\end{equation}
We also call \eqref{weyl-1} Weyl's inequality,
since it is equivalent to \eqref{wi<}.
Indeed,
assume that \eqref{weyl-1} holds for any $A$ and $B$.
Noting $n_+(B-\la_j(B)I)\le j-1$, we have
$$\la_{i+j-1}(A+B)=\la_{i+j-1}(A+B-\la_j(B)I)+\la_j(B)
\le\la_{i}(A)+\la_{j}(B)$$
by \eqref{weyl-1}.
Such a form \eqref{weyl-1} of Wely's inequality is often more convenient to use.
For example,
if $B$ is positive semi-definite, i.e., $n_-(B)=0$,
then $\la_i(A)\le\la_i(A+B)$ by
\eqref{weyl-1},
which is the monotonicity theorem.
Further, if $B=\alpha\alpha^*$ for some column vector $\alpha\in\mathbb{C}^n$,
then $n_+(B)\le 1$ and $n_-(B)=0$, and so
$$\la_n(A)\le\la_n(A+B)\le\la_{n-1}(A)\le\cdots\le\la_2(A+B)\le\la_1(A)\le\la_1(A+B),$$
which is the interlacing theorem for a rank-one Hermitian perturbation of a Hermitian matrix.
It is well known that the converse of this interlacing theorem is true
(see \cite[Theorems 4.3.26]{HJ13} for instance or Lemma \ref{10}).
In this paper we consider a more general problem: the converse of Weyl's inequality.
We first introduce some definitions and notations.

Let $\R$ ($\R_n$, resp.) denote the set of real polynomials (of degree $n$, resp.)
with only real roots and with positive leading coefficients.
In particular, let $\RM$ denote the set of monic polynomials in $\R$.
For $g\in\R$, we use $r_i(g)$ denote its roots and arrange them in non-increasing order:
$r_1(g)\ge r_2(g)\ge\cdots\ge r_n(g)$.
For convenience,
set $r_i(g)=+\infty$ for $i<1$ and $r_i(g)=-\infty$ for $i>\deg g$.

\begin{definition}
Let $f,g\in\R$ and $p,q\in\mathbb{N}$.
The polynomial $f$ is said to {\it $(p,q)$-interlace} the polynomial $g$,
denoted by $f\bint{p}{q} g$,
if
\begin{equation*}\label{pq-int}
r_{i+p}(g)\le r_{i}(f)\le r_{i-q}(g)
\end{equation*}
for all $i\in\mathbb{Z}$.
\end{definition}

The following properties are immediate from the definition.

\begin{prop}\label{fact}
\begin{enumerate}[\rm (a)]
  \item $f\bint{p}{q} f$ for any $p$ and $q$.
  \item $f\bint{p}{q} g$ is equivalent to $g\bint{q}{p} f$.
  \item $f\bint{p}{q} h$ and $h\bint{s}{t} g$ imply that $f\bint{p+s}{q+t} g$.
  \item $f\bint{p}{q} g$ implies that $f\bint{s}{t} g$ for any $s\ge p$ and $t\ge q$.
  \item $f\bint{p}{q} g$ implies that $-p\le \deg f-\deg g\le q$.
\end{enumerate}
\end{prop}

There are two particular interesting special cases in the definition.
When $f\bint{1}{0} g$, we say that $f$ {\it interlaces} $g$ and denote by $f\intrz g$. 
It is well known that $f(x)$ and $g(x)$ are interlacing ($f\intrz g$ or $g\intrz f$)
if and only if for any $a,b\in\mathbb{R}$,
all roots of the polynomial $af(x)+bg(x)$ are real.
(see Obreschkoff \cite[Satz 5.2]{Obr63} for instance).
When $f\bint{1}{1} g$, we say that $f$ and $g$ are {\it compatible} and denote simply by $f\sim g$.
Chudnovsky and Seymour \cite[Theorem 3.4]{CS07} showed that $f$ and $g$ are compatible if and only if for any $a,b\ge 0$,
all roots of the polynomial $af(x)+bg(x)$ are real.
Compatible and interlacing properties of polynomials
are often encountered in combinatorics \cite{CS07,Fis08,LW07rz,MSS15,SV15,WYjcta05,YZ17}.

Let $A$ and $B$ be two
Hermitian matrices.
Denote $A\bint{p}{q} B$
if their characteristic polynomials $\det (\la I-A)\bint{p}{q}\det (\la I-B)$.
Using this notation, Weyl's inequality \eqref{weyl-1} can be restated as follows.

\begin{WI}
Let $A$ and $B$ be two Hermitian matrices of the same order.
Assume that $n_+(B-A)\le p$ and $n_-(B-A)\le q$. Then
$A\bint{p}{q} B$.
\end{WI}

For $f\in\RM$,
denote by $\H(f)$ the set of Hermitian matrices with characteristic polynomial $f$.
The objective of this note is to establish the converse of Weyl's inequality.

\begin{thm}\label{main-thm}
Let $f,g\in\RM$ have the same degree and $f\bint{p}{q} g$.
Then there exist $A\in\H(f)$ and $B\in\H(g)$ such that
$n_+(B-A)\le p$ and $n_-(B-A)\le q$.
\end{thm}

In the next section we investigate the $(p,q)$-interlacing property.
Some known results about the interlacing and compatible polynomials
will be extended to $(p,q)$-interlacing polynomials.
As an application we prove Theorem \ref{main-thm}.
In \S 3 we discuss some results closely related to Theorem \ref{main-thm}.
These results can also be obtained from the same approach used in \S 2.

\section{$(p,q)$-interlacing property and proof of the theorem}

Let $\n(f,r)$ be the number of real roots of $f(x)$ in the interval $[r,+\infty)$.
It is well known that $f$ interlaces $g$ if and only if
$\n(f,r)\le\n(g,r)\le \n(f,r)+1$ for any $r\in\mathbb{R}$.
Chudnovsky and Seymour \cite[Theorem 3.4]{CS07} showed that $f$ and $g$ are compatible if and only if
$|\n(g,r)-\n(f,r)|\le 1$ for any $r\in\mathbb{R}$.
A common generalization of these two results is the following.

\begin{lem}\label{nfr-ngr}
Suppose that $f,g\in\R$ and $p,q\in\mathbb{N}$.
Then $f\bint{p}{q} g$ if and only if
$-p\le\n(f,r)-\n(g,r)\le q$ for any $r\in\mathbb{R}$.
\end{lem}
\begin{proof}
Assume that $f\bint{p}{q} g$.
Let $r\in\mathbb{R}$ and $\n(f,r)=i$.
Then $r_{i-q}(g)\ge r_i(f)\ge r$.
Thus $\n(g,r)\ge i-q=\n(f,r)-q$.
Similarly,
$g\bint{q}{p} f$ implies that $\n(f,r)\ge\n(g,r)-p$.
We conclude that $-p\le\n(f,r)-\n(g,r)\le q$.

Conversely, assume that $-p\le\n(f,r)-\n(g,r)\le q$ for any $r\in\mathbb{R}$.
Then $\n(g,r_i(f))\ge\n(f,r_i(f))-q\ge i-q$,
and so $r_{i-q}(g)\ge r_i(f)$.
Similarly, $r_{i-p}(f)\ge r_i(g)$.
Thus $r_{i+p}(g)\le r_i(f)\le r_{i-q}(g)$, i.e.,
$f\bint{p}{q} g$.
\end{proof}

\begin{coro}\label{ppq}
Let $f,g,h\in\R$.
Then $f\bint{p}{q} g$ if and only if 
$(fh)\bint{p}{q} (gh)$.
\end{coro}

We next show that the converse of Proposition~\ref{fact}~(c) is also true.

\begin{lem}\label{pqst}
Suppose that $f\bint{p}{q} g$.
For $0\le s\le p$ and $0\le t\le q$,
denote $k=\max\{\deg f-t,\deg g-p+s\}$ and $m=\min\{\deg f+s,\deg g+q-t\}$.
Then for each integer $d\in [k,m]$,
there exists a real-rooted polynomial $h$ of degree $d$
such that $f\bint{s}{t} h$ and $h\bint{p-s}{q-t} g$.
\end{lem}
\begin{proof}
By the definition of $f\bint{p}{q} g$, we have
$$
r_{i+t}(f)\leq r_{i-q+t}(g),\quad r_{i+p-s}(g)\leq r_{i-s}(f),
$$
and so
$$\max\left\{r_{i+t}(f),r_{i+p-s}(g)\right\}\leq \min\left\{r_{i-s}(f),r_{i-q+t}(g)\right\}.$$
For $i=1,2,\ldots,m$,
denote $$a_i=\max\{r_{i+t}(f),r_{i+p-s}(g)\},\quad b_i=\min\{r_{i-s}(f),r_{i-q+t}(g)\}$$
and $I_i=[a_i,b_i]$ ($I_i=(a_i,b_i]$ if $a_i=-\infty$ or $I_i=[a_i,b_i)$ if $b_i=+\infty$).
Clearly, $(a_i)$ is non-increasing.
Hence we may choose $r_i\in I_i$ such that $(r_i)$ is non-increasing.
Now for each $d\in [k,m]$, define a polynomial $h(x)=\prod_{i=1}^{d}(x-r_i)$.
Then $r_i(h)=r_i\in I_i$ for $i=1,2,\ldots,d$, and so
$$
r_{i+t}(f)\leq r_i(h) \leq r_{i-s}(f),\quad
r_{i+p-s}(g) \leq r_i(h) \leq r_{i-q+t}(g).
$$
Thus $f\bint{s}{t} h$ and $h\bint{p-s}{q-t} g$.
\end{proof}

\begin{coro}[{\cite[Theorem 3.5]{CS07}}]
Suppose that $f,g\in\R$.
Then $f\sim g$ if and only if
there exists
$h\in\R$ such that $h\intrz f$ and $h\intrz g$.
\end{coro}

It is easy to see that if $f\intrz g$,
then $f'\intrz g'$.
Also, Chudnovsky and Seymour \cite[Theorem 3.1]{CS07} showed that
if $f\sim g$,
then $f'\sim g'$.
We have the following more general result.

\begin{coro}
If $f\bint{p}{q} g$,
then the derivative $f'\bint{p}{q} g'$.
\end{coro}
\begin{proof}
We proceed by induction on $p+q$.
The statement for $p+q=1$ is just the interlacing case.
Suppose that $p+q>1$ and $f\bint{p}{q} g$.
We may assume, without loss of generality, that $p\ge 1$.
Then by Lemma \ref{pqst},
there exists $h$ such that $f\bint{p-1}{q} h$ and $h\bint{1}{0} g$.
By the inductive hypothesis,
$f'\bint{p-1}{q} h'$ and $h'\bint{1}{0} g'$.
Thus $f'\bint{p}{q} g'$ by Proposition~\ref{fact}~(c).
\end{proof}

For convenience,
we introduce an abbreviative notation.
If a Hermitian matrix $A=\sum_{i=1}^{p}{\alpha_i}{\alpha_i}^*$,
where $\alpha_i$ are $p$ (zero or nonzero) complex vectors,
then we denote simply by $A=\s{p}$.
Also, denote $\s{0}=0$.
Obviously, $\s{p}$ is positive semi-definite of rank at most $p$.
It is also clear that $\s{p}+\s{q}=\s{p+q}$ and
$U\left(\s{p}\right)U^*=\s{p}$ for any matrix $U$.

\begin{lem}\label{pqm}
Let $H$ be a Hermitian matrix.
Then $n_+(H)\le p$ and $n_-(H)\le q$ if and only if $H=\s{p}-\s{q}$.
\end{lem}
\begin{proof}
Clearly,
$$n_+\left(\s{p}-\s{q}\right)\le n_+\left(\s{p}\right)\le\text{rk}\left(\s{p}\right)\le p$$
and
$$n_-\left(\s{p}-\s{q}\right)=n_+\left(\s{q}-\s{p}\right)\le q.$$

Conversely, if $n_+(H)\le p$ and $n_-(H)\le q$,
then by the spectral decomposition and Sylvester's law of inertia for Hermitian matrices,
$$H=\sum_{i}\la_i(H)P_iP_i^*=\s{p}-\s{q},$$
where $P_i$ are the corresponding orthonormal eigenvectors to $\la_i(H)$.
\end{proof}

The following folklore result is the converse of the interlacing theorem
for a rank-one Hermitian perturbation of a Hermitian matrix
(see \cite[Theorems 4.3.26]{HJ13} for instance).
We give a direct proof of it for completeness.

\begin{lem}\label{10}
Let $f,g\in\RM_n$ and $f\intrz g$.
Then there exist a Hermitian matrix $A$ and a complex vector $\alpha$
such that $A\in\H(f)$ and $A+\alpha\alpha^*\in\H(g)$.
\end{lem}
\begin{proof}
Consider first the special case that $f$ and $g$ are coprime.
Then $f$ has only simple roots by $f\intrz g$.
Denote $r_i=r_i(f)$ and $f_i(x)=\frac{f(x)}{x-r_i}$.
Then $f(x),f_1(x),\ldots,f_n(x)$
form a basis of the vector space $\mathbb{R}[x]_{n}$ of real polynomials with
degree less than $n+1$.
Let
$g(x)=f(x)+\sum_{i=1}^nc_if_i(x)$.
Then $g(r_i)=c_if_i(r_i)$.
Note that $\sgn [f_i(r_i)]=(-1)^{i-1}$ and $\sgn [g(r_i)]=(-1)^{i}$.
Hence $c_i<0$.
Define $A=\diag(r_1,\ldots,r_n)$ and $\alpha=(a_1,\ldots,a_n)^T$,
where $a_i=\sqrt{-c_i}$.
Then $\det (xI-A)=f(x)$.
On the other hand, we have 
\begin{equation*}
\det (xI-A-\alpha\alpha^*)=\det (xI-A)-\sum_{i=1}^na_i^2\prod_{j\neq i}(x-r_j)=g(x).
\end{equation*}
Thus $A\in\H(f)$ and $A+\alpha\alpha^*\in\H(g)$.

Consider next the general case.
Let $f=(f,g)f_1$ and $g=(f,g)g_1$.
Then $f_1\intrz g_1$ and $(f_1,g_1)=1$.
Thus there are two Hermitian matrices $A_1\in\H(f_1)$ and $A_1+\alpha_1\alpha_1^*\in\H(g_1)$.
Assume that $(f,g)=\prod_{j=1}^m(x-s_j)$ and define
$$A=\left[\begin{array}{cccc}
A_1 & & & \\
& s_1 &&  \\
&&\ddots&\\
&&& s_m\\
\end{array}\right],\quad
\alpha=\left[\begin{array}{c}
\alpha_1 \\
0 \\
\vdots \\
0 \\
\end{array}\right].
$$
Then $A\in\H(f)$ and $A+\alpha\alpha^*\in\H(g)$.
This completes the proof.
\end{proof}

\begin{lem}\label{p0-lem}
Let $f,g\in\RM_n$ and $f\bint{p}{0} g$.
Then there exist two Hermitian matrices $A\in\H(f)$ and $B\in\H(g)$
such that $B-A=\s{p}$.
\end{lem}
\begin{proof}
We proceed by induction on $p$.
The statement for case $p=1$ follows from Lemma \ref{10}.
Suppose now that $p>1$ and $f\bint{p}{0} g$.
Then by Lemma \ref{pqst},
these exists $h\in\RM_n$
such that $f\bint{p-1}{0} h$ and $h\bint{1}{0} g$.
By $f\bint{p-1}{0} h$ and the induction hypothesis,
these exist $A\in\H(f)$ and $C\in\H(h)$
such that $C-A=\s{p-1}$.
By $h\bint{1}{0} g$ and Lemma \ref{10},
these exist $B_1\in\H(g)$ and $C_1\in\H(h)$
such that $B_1-C_1=\s{1}$.
Since two Hermitian matrices $C$ and $C_1$ have the same characteristic polynomial,
there is a unitary matrix $U$ such that $UC_1U^*=C$.
Define $B=UB_1U^*$.
Then $B\in\H(g)$ and
$$B-C=U(B_1-C_1)U^*=U\left(\s{1}\right)U^*=\s{1}.$$
It follows that
$$B-A=(B-C)+(C-A)=\s{1}+\s{p-1}=\s{p},$$
as required.
Thus the proof is complete by induction.
\end{proof}

We are now in a position to prove the theorem.

\begin{proof}[Proof of Theorem \ref{main-thm}]
By Lemma \ref{pqm},
it suffices to show that if $f\bint{p}{q} g$,
then there exist two Hermitian matrices $A\in\H(f)$ and $B\in\H(g)$
such that $B-A=\s{p}-\s{q}$.

If $p=0$ or $q=0$, then the statement follows from Lemma \ref{p0-lem}.
Assume next that $p,q>0$.
Then by Lemma \ref{pqst},
there exists $h\in\RM_n$
such that $f\bint{p}{0} h$ and $g\bint{q}{0} h$.
By $f\bint{p}{0} h$ and Lemma \ref{p0-lem},
there exist $A\in\H(f)$ and $C\in\H(h)$ such that $C-A=\s{p}$.
By $g\bint{q}{0} h$ and Lemma \ref{p0-lem},
there exist $B_1\in\H(g)$ and $C_1\in\H(h)$ such that $C_1-B_1=\s{q}$.
Since two Hermitian matrices $C$ and $C_1$ have the same characteristic polynomial,
there is a unitary matrix $U$ such that $UC_1U^*=C$.
Define $B=UB_1U^*$.
Then $B\in H(g)$ and
$$C-B=U(C_1-B_1)U^*=U\left(\s{q}\right)U^*=\s{q}.$$
It follows that
$$B-A=(C-A)-(C-B)=\s{p}-\s{q},$$
as required.
This completes the proof of the theorem.
\end{proof}

\begin{rem}
We can also prove Theorem \ref{main-thm} directly from Lemmas \ref{pqst} and \ref{10} by induction on $p+q$.
\end{rem}

\section{Remarks}

Let $\alpha=(a_1,\ldots,a_k)$ and $\beta=(b_1,\ldots,b_k)$ be two real vectors whose entries arranged in non-increasing order.
Denote $\alpha\le\beta$ if $\beta-\alpha$ is non-negative,
and further denote $\alpha\le_r\beta$ if $\beta-\alpha$ has at least $r$ positive entries.
Let $f$ and $g$ be two monic real-rooted polynomials of degree $n$.
By the definition, $f\bint{p}{q} g$ is equivalent to
$(r_{q+1}(f),\ldots,r_n(f))\le (r_1(g),\ldots,r_{n-q}(g))$
and
$(r_{p+1}(g),\ldots,r_n(g))\le (r_1(f),\ldots,r_{n-p}(f))$.
Denote $f\bb{p}{q} g$ if
$$(r_{q+1}(f),\ldots,r_n(f))\le_p (r_1(g),\ldots,r_{n-q}(g))$$
and
$$(r_{p+1}(g),\ldots,r_n(g))\le_q (r_1(f),\ldots,r_{n-p}(f)).$$
The notation $\bb{p}{q}$ enjoys some properties similar to $\bint{p}{q}$.
For example,
\begin{enumerate}[\rm (i)]
  \item If $f\bb{p}{q} g$,
   then there exists $h$ such that $f\bb{p}{0} h$ and $g\bb{q}{0} h$.
  \item If $f\bb{p}{0} g$,
   then there exists $h$ such that $f\bb{p-1}{0} h$ and $h\intrz g$.
\end{enumerate}
To prove them, it suffices to choose $r_i(h)$ in the proof of Lemma \ref{pqst}
as far as possible within the interval $I_i$ unless $I_i$ consists of only one point.

Theorem~\ref{main-thm} states that if $f\bint{p}{q} g$,
then there exist two Hermitian matrices $A\in\H(f)$ and $B\in\H(g)$
such that $n_+(B-A)\le p$ and $n_-(B-A)\le q$.
Given $0\le s\le p$ and $0\le t\le q$,
a natural problem is when there exist two Hermitian matrices
$A\in\H(f)$ and $B\in\H(g)$ such that
$n_+(B-A)=s$ and $n_-(B-A)=t$.
Li and Poon \cite[Theorem 2.1]{LP10} gave the characterization for the case $s=p$ and $t=q$,
which can also be proved by a similar technique used in the previous section.
We omit the details for the sake of simplicity.

\begin{thm}[{\cite[Theorem 2.1]{LP10}}]\label{thm-lp12}
Let $f,g\in\RM_n$ and $f\bint{p}{q} g$.
Suppose that $p+q\le n$ and $f\bb{p}{q} g$.
Then there exist $A\in\H(f)$ and $B\in\H(g)$ such that
$n_+(B-A)=p$ and $n_-(B-A)=q$.
\end{thm}

Another result closely related to Theorem \ref{main-thm} is Cauchy's interlacing theorem,
which states that
a Hermitian matrix $A$ interlaces its bordered Hermitian matrix
$\left[
   \begin{array}{cc}
     A & \alpha \\
     \alpha^* & a \\
   \end{array}
 \right]$
(see \cite[Theorem 4.3.17]{HJ13} for instance).
More generally,
the inclusion principle states that two Hermitian matrices
$A\bint{p}{0}\left[
   \begin{array}{cc}
     A & B \\
     B^* & C \\
   \end{array}
 \right]$,
where $C$ is a $p\times p$ Hermitian matrix
(see \cite[Theorem 4.3.28]{HJ13} for instance).
Fan and Pall \cite[Theorem 1]{FP57} established the following converses of the inclusion principle.

\begin{thm}[{\cite[Theorem 1]{FP57}}]\label{cip}
Let $f$ and $g$ be two monic real-rooted polynomials
satisfying $\deg g=\deg f+p$ and $f\bint{p}{0} g$.
Then there are two Hermitian matrices $A$ and
$\left[
   \begin{array}{cc}
     A & B \\
     B^* & C \\
   \end{array}
 \right]$
such that their characteristic polynomials are $f$ and $g$ respectively.
\end{thm}

We can prove the converses of Cauchy's interlacing theorem
by the same technique used in the proof of Lemma \ref{10},
and then prove the converses of the inclusion principle by induction on $p$.

\section*{Acknowledgement}

This work was supported partially by the National Natural Science Foundation of China
(No. 11771065).
The authors thank the anonymous referee for his/her careful reading and kind comments.


\end{document}